\documentclass[12pt]{article}

\usepackage{amssymb}   
\usepackage{amsfonts}  
\usepackage{amsmath}   

\usepackage{tikz}
\usetikzlibrary{calc}
\usetikzlibrary{decorations.pathreplacing}
\usetikzlibrary{shapes,snakes}

\usepackage{setspace}
\usepackage{subfigure}
\usepackage{url}
\usepackage{graphicx}
\usepackage{booktabs}
\usepackage{ifthen}
\usepackage{enumerate}
\usepackage{comment}

\usepackage[margin=1.25in]{geometry}

\newtheorem{thm}{Theorem}[section]

\newtheorem{ex}[thm]{Example}
\numberwithin{equation}{section}

\begin{document}

\title{On the Edge-Balanced Index Sets of Complete Even Bipartite Graphs} 

\author{
Ha Dao\thanks{Undergraduate Student, Clayton State University, (\texttt{hdao@student.clayton.edu})}
\and
Hung Hua\thanks{Undergraduate Student, Georgia Institute of Technology, (\texttt{hhua7@gatech.edu})}
\and
Michael Ngo\thanks{Undergraduate Student, Clayton State University, (\texttt{mngo4@student.clayton.edu})}
\and
Christopher Raridan\thanks{Department of Mathematics, Clayton State University, (\texttt{ChristopherRaridan@clayton.edu})}
}
   
\maketitle

\begin{abstract}
In 2009, Kong, Wang, and Lee introduced the problem of finding the edge-balanced index sets ($EBI$) of complete bipartite graphs $K_{m,n}$, where they examined the cases $n=1$, $2$, $3$, $4$, $5$ and the case $m=n$. Since then the problem of finding $EBI(K_{m,n})$, where $m \geq n$, has been completely resolved for the $m,n=\text{odd, odd}$ and odd, even cases. In this paper we find the edge-balanced index sets for complete bipartite graphs where both parts have even cardinality. 
\\[5pt] 
\\[5pt]
\end{abstract}

\section{Introduction}
Given a graph $G$, let $V$ and $E$ denote the vertex set and edge set of $G$, respectively. A \textit{binary edge-labeling} of $G$ is a function $f : E \to \{0,1\}$. For $i \in \{0,1\}$, we call $e \in E$ an \textit{$i$-edge} if $f(e)=i$. Let $e(i)$ denote the number of $i$-edges under a binary edge-labeling $f$. If $\lvert e(1)-e(0)\rvert \leq 1$, we say that $f$ is \textit{edge-friendly}. Under $f$, we let $\deg_i(v)$ denote the number of $i$-edges incident with $v \in V$. If $f$ is edge-friendly, then $f$ will induce a (possibly partial) \textit{vertex-labeling} where $v$ is labeled $1$ if $\deg_1(v) > \deg_0(v)$, labeled $0$ if $\deg_0(v) > \deg_1(v)$, and unlabeled if $\deg_1(v)=\deg_0(v)$. Any vertex labeled $i$ is called an \textit{$i$-vertex} and let $v(i)$ denote the number of $i$-vertices under an edge-friendly labeling $f$. The \textit{edge-balanced index set} of $G$ is defined as 
\[ EBI(G) = \{\lvert v(1)-v(0) \rvert : \text{$f$ is edge-friendly}\} \]
and an element in $EBI(G)$ will be called a \textit{balanced index}.
More information about graph labelings, including many results concerning edge-friendly labelings, can be found in Gallian's dynamic survey~\cite{GallianYYYY}.

Kong, Wang, and Lee~\cite{KWL2009} explored the problem of finding the edge-balanced index sets of complete bipartite graphs $K_{m,n}$ by investigating the cases where $n=1$, $2$, $3$, $4$, and $5$, as well as the case where $m=n$. In~\cite{KMPR2014}, Krop, Minion, Patel, and Raridan concluded the $EBI$ problem for complete bipartite graphs with both parts of odd cardinality (the ``odd/odd'' case). The following year, Hua and Raridan~\cite{HR2014} found $EBI(K_{m,n})$ where $m > n$, $m$ is odd and $n$ is even (the ``odd/even'' case). In this paper we find the edge-balanced index sets for complete even bipartite graphs (the ``even/even'' case).

\section{Notations and Conventions}
Throughout the rest of this paper, we let $K_{m,n}$ denote a complete bipartite graph with part $A = \{ v_1, v_2, \dots, v_m \}$ of even cardinality $m$ and part $B = \{ u_1, u_2, \dots, u_n \}$ of even cardinality $n$, where $m \geq n \geq 2$. For any edge-friendly labeling of $K_{m,n}$, we have that $e(0) = e(1) = \frac{mn}{2}$. Without loss of generality, we may assume that labelings are chosen so that $v(1) \geq v(0)$. Let $v_A(i)$ and $v_B(i)$ represent the number of $i$-vertices in parts $A$ and $B$, respectively, and note that $v(i)=v_A(i)+v_B(i)$. 

For integers $a < b$ define $[a,b] = \{a, a+1, \dots, b\}$ and $[a,a] = \{a\}$. If $a$ is positive, $[a] = \{1, 2, \dots, a\}$; otherwise  $[a] = \emptyset$. We organize the edge labels of an edge-friendly labeling of $K_{m,n}$ as an $n \times m$ binary matrix whose $(s,t)$-entry is the label on edge $u_s v_t$, where $s \in [n]$ and $t \in [m]$. The vertex label for $v_t \in A$ or $u_s \in B$ is found by comparing the sum of the entries in column $t$ or row $s$ with $\frac{n}{2}$ or $\frac{m}{2}$, respectively. 

\begin{ex}
\label{ex:K-4-4}
Finding the corresponding balanced index for each of the following edge-friendly labelings of $K_{4,4}$ is straightforward:~(a) and~(b) show two different edge-friendly labelings that each give $0 \in EBI(K_{4,4})$, (c) shows that $1 \in EBI(K_{4,4})$, and (d) shows that $2 \in EBI(K_{4,4})$.  
\begin{figure}[ht]
\centering
\subfigure[]
{
\begin{tabular}{cccc} 
1&1&1&1 \\
1&1&1&1 \\
0&0&0&0 \\
0&0&0&0
\end{tabular}
}
\hspace{0.25in}
\subfigure[]
{
\begin{tabular}{cccc} 
1&1&1&1 \\
1&1&1&0 \\
1&0&0&0 \\
0&0&0&0
\end{tabular}
}
\hspace{0.25in}
\subfigure[]
{
\begin{tabular}{cccc} 
1&0&1&1 \\
1&1&1&0 \\
1&1&0&0 \\
0&0&0&0
\end{tabular}
}
\hspace{0.25in}
\subfigure[]
{
\begin{tabular}{cccc} 
1&1&1&0 \\
1&1&1&0 \\
1&1&0&0 \\
0&0&0&0
\end{tabular}
}
\caption{Some edge-friendly labelings of $K_{4,4}$.}
\end{figure}
\end{ex}

The quotient and remainder when $x$ is divided by $y$ using the division algorithm is denoted by $x~\text{div}~y$ (or, $\lfloor \frac{x}{y} \rfloor$) and $x \bmod y$, respectively.

\section{Finding $EBI(K_{m,n})$}
In this section, we prove
\begin{thm}
\label{thm:main}
Let $K_{m,n}$ be a complete bipartite graph with parts of cardinality $m$ and $n$, where $m \geq n$ are positive even integers. Then $EBI(K_{m,2}) = \{ 0 \}$. For $n \geq 4$, let $k = \left\lfloor \frac{mn}{n+2} \right\rfloor$, $k' = \frac{mn}{2} \bmod \left( \frac{n}{2}+1 \right)$, $j = \left\lfloor \frac{mn}{m+2} \right\rfloor$, and $j' = \frac{mn}{2} \bmod \left( \frac{m}{2}+1 \right)$. Then  
\[
EBI(K_{m,n}) =
\begin{cases}
\{ 0,1, \dots, 2(k+j)+2-m-n \}, &\text{if~$k' = \frac{n}{2}$ and $j' = \frac{m}{2}$}, \\
\{ 0,1, \dots, 2(k+j)+1-m-n \}, &\text{if either~$k' = \frac{n}{2}$ or $j' = \frac{m}{2}$}, \\
\{ 0,1, \dots, 2(k+j)-m-n \}, &\text{if~$k' < \frac{n}{2}$ and $j' < \frac{m}{2}$}.
\end{cases}
\]
\end{thm}

\begin{proof}
In~\cite{KWL2009}, the authors show that $EBI(K_{m,2}) = \{ 0 \}$ for all integers $m \geq 2$. Throughout the rest of this paper, we assume that $m \geq n \geq 4$ are both even.

To find the maximal element of $EBI(K_{m,n})$ we need an edge-friendly labeling that maximizes the value of $v(1)$ while at the same time minimizes the value of $v(0)$. Let $k$ and $j$ represent the maximum value of $v_A(1)$ and $v_B(1)$, respectively. A $1$-vertex $v \in A$ must have $\deg_1(v) \geq \frac{n}{2}+1$, so the maximum value of $v_A(1)$ is $k = e(1)~\text{div}~\left( \frac{n}{2}+1 \right) = \left\lfloor \frac{mn}{n+2} \right\rfloor$. Given any edge-friendly labeling that maximizes $v_A(1)$ where each of the $1$-vertices in $A$ has exactly $\left( \frac{n}{2}+1 \right)$ $1$-edges, the number of $1$-edges incident with the other $(m-k)$ vertices in $A$ is $k' = e(1) \bmod \left( \frac{n}{2}+1 \right)$. If $v_A(1)$ is maximized and $k' < \frac{n}{2}$, there are not enough of these extra $1$-edges to create an unlabeled vertex in $A$, so $v_A(0) = m-k$. If instead we have that $k' = \frac{n}{2}$, there are enough extra $1$-edges to allow an unlabeled vertex in part $A$ and having an unlabeled vertex in $A$ reduces the value of $v_A(0)$, which would be $m-k-1$ in this case. Similarly, the maximum value of $v_B(1)$ is $j = e(1)~\text{div}~\left( \frac{m}{2}+1 \right) = \left\lfloor \frac{mn}{m+2} \right\rfloor$. For any edge-friendly labeling that maximizes $v_B(1)$ where each of the $1$-vertices in $B$ has exactly $\left( \frac{m}{2}+1 \right)$ $1$-edges, the number of $1$-edges incident with the other $(n-j)$ vertices in $B$ is $j' = e(1) \bmod \left(\frac{m}{2}+1\right)$. When $v_B(1)$ is maximized, if $j' < \frac{m}{2}$, then $v_B(0) = n-j$, and if $j' = \frac{m}{2}$, then $v_B(0) = n-j-1$. 

Now, we need to find an edge-friendly labeling that maximizes both $v_A(1)$ and $v_B(1)$ at the same time, thus maximizing their sum $v(1)$. Maximizing $v(1)$ and allowing part $A$ or part $B$ to contain an unlabeled vertex (when $k' = \frac{n}{2}$ or $j' = \frac{m}{2}$) minimizes both $v_A(0)$ and $v_B(0)$ at the same time, thus minimizing their sum $v(0)$. That is,
\[ 
\max EBI(K_{m,n}) =
\begin{cases}
2(k+j)+2-m-n, &\text{if~$k' = \frac{n}{2}$ and $j' = \frac{m}{2}$}, \\
2(k+j)+1-m-n, &\text{if either~$k' = \frac{n}{2}$ or $j' = \frac{m}{2}$}, \\
2(k+j)-m-n, &\text{if~$k' < \frac{n}{2}$ and $j' < \frac{m}{2}$}.
\end{cases}
\]

For example, if $m=n=4$, then $k = j = k' = j' = \frac{m}{2} = \frac{n}{2} = 2$ and $\max EBI(K_{4,4}) = 2$. Example~\ref{ex:K-4-4}(d) shows an edge-friendly labeling for $K_{4,4}$ that produces the maximal balanced index for this graph. For all other values of even integers $m \geq n$, it follows that $k > \frac{m}{2}$ and $j > \frac{n}{2}$, which ensures that in each of the three cases above, $\max EBI(K_{m,n})$ is a positive quantity.

We now discuss an algorithm that provides a sequence of edge-friendly labelings (actually, a sequence of edge-label switches) that correspond to each of the balanced indices from $0$ to $\max EBI(K_{m,n})$. For each of the following steps, we mention only the vertices whose labels have changed due to the switches described in that step. For some values of $m$ and $n$, running the entire algorithm is unnecessary; indeed, the procedure should be terminated when $\max EBI(K_{m,n})$ has been obtained. We will provide a few example graphs when early termination is allowed.  

\textbf{Step~0.} For $s \in \left[ \frac{n}{2} \right]$ and $t \in \left[ m \right]$, create an $n \times m$ matrix whose $(s,t)$-entry is $1$ and set all other entries to $0$. The top half of this matrix is all 1s and the bottom half is all 0s, so $v_A(1) = v_A(0) = 0$, $v_B(1) = v_B(0) = \frac{n}{2}$, and $0 \in EBI(K_{m,n})$.

In Steps~1-3, we let $q(a)$ and $r(a)$ represent the quotient and remainder, respectively, when $a-1$ is divided by $\frac{n}{2}$ using the division algorithm. 

\textbf{Step~1.} For $a \in \left[ \frac{m}{2}-1 \right]$, switch the $\left( \frac{n}{2}+1, a \right)$-entry with the $\left( \frac{n}{2} - r(a), m - q(a) \right)$-entry. Here, each edge-label switch exchanges a $0$ in row $\frac{n}{2}+1$ with a $1$ in the last $q\left( \frac{m}{2}-1 \right)+1$ columns and above the $\left( \frac{n}{2}+1 \right)$-st row. Such switches cause $v_a$ to become a $1$-vertex for $a \in \left[ \frac{m}{2}-1 \right]$ and $v_{m+1-b}$ to become a $0$-vertex for $b \in \left[ q\left( \frac{m}{2}-1 \right)+1 \right]$. Note that the first switch of a $1$ in each column has no effect on the balanced index since both $v_A(1)$ and $v_A(0)$ increase by $1$, but that each subsequent switch will increase the balanced index by $1$. At the end of Step~1, $\deg_0(u_{\frac{n}{2}+1}) = \frac{m}{2}+1$, which means $u_{\frac{n}{2}+1}$ is just ``barely'' a $0$-vertex.

\textbf{Step~2.} Switch the $\left( \frac{n}{2}+1, \frac{m}{2} \right)$-entry with the $\left( \frac{n}{2}-r\left( \frac{m}{2} \right), \frac{m}{2} \right)$-entry. This switch causes vertex $u_{\frac{n}{2}+1}$ to become an unlabeled vertex so the balanced index increases by $1$. Now, switch the $\left( \frac{n}{2}-r\left( \frac{m}{2} \right), \frac{m}{2} \right)$-entry with the $\left( \frac{n}{2}-r\left( \frac{m}{2} \right), m-q\left( \frac{m}{2} \right) \right)$-entry, which causes $v_{\frac{m}{2}}$ to become a $1$-vertex. If $r\left( \frac{m}{2} \right) = 0$, or equivalently $m=tn+2$ for some integer $t \geq 1$, then this switch also causes $v_{m-q\left( \frac{m}{2} \right)}$ to become a $0$-vertex and there is no change in the balanced index; otherwise, $v_{m-q\left( \frac{m}{2} \right)}$ was already a $0$-vertex and the balanced index increases by $1$. Note that for $m=n=4$, we terminate the procedure since $k = \frac{m}{2}$ and $j = \frac{n}{2}$ for this case. 

For other values of $m$ and $n$, the $\left( \frac{m}{2}+1 \right)$-st (double) switch is similar to the $\frac{m}{2}$-th. Exchange the $\left( \frac{n}{2}+1, \frac{m}{2}+1 \right)$-entry with the $\left( \frac{n}{2}-r\left( \frac{m}{2}+1 \right), \frac{m}{2}+1 \right)$-entry. This switch causes $u_{\frac{n}{2}+1}$ to become a $1$-vertex so the balanced index increases by $1$. Now, switch the $\left( \frac{n}{2}-r\left( \frac{m}{2}+1 \right), \frac{m}{2}+1 \right)$-entry with the $\left( \frac{n}{2}-r\left( \frac{m}{2}+1 \right), m-q\left( \frac{m}{2}+1 \right) \right)$-entry, which causes $v_{\frac{m}{2}+1}$ to become a $1$-vertex. If $r\left( \frac{m}{2}+1 \right) = 0$, or equivalently $m=tn$ for some integer $t \geq 1$, then this switch also causes $v_{m-q\left( \frac{m}{2}+1 \right)}$ to become a $0$-vertex and there is no change in the balanced index. Otherwise, $v_{m-q\left( \frac{m}{2}+1 \right)}$ was already a $0$-vertex and the balanced index increases by $1$. Note that if $(m,n) = (6,4)$, $(8,4)$, $(10,4)$, or $(6,6)$, then we terminate the procedure since $v_A(1) = k = \frac{m}{2}+1$, $v_B(1) = j = \frac{n}{2}+1$, and $j' < \frac{m}{2}$ since maximizing $v_B(1)$ does not allow for an unlabeled vertex in part $B$ for these cases.

\textbf{Step~3.} Perform this step if and only if $k > \frac{m}{2}+1$. For $a \in \left[ \frac{m}{2}+2,k \right]$, switch the $\left( \frac{n}{2}+1, a \right)$-entry with the $\left( \frac{n}{2} - r(a), m - q(a) \right)$-entry. This step is essentially the same as Step~1, just applied to a different set of indices. Moreover, we have now forced $v_A(1) = k$ and either $v_A(0) = m-k$ (there are no unlabeled vertices in part $A$) or $v_A(0) = m-k-1$ (there is one unlabeled vertex in $A$). That is, we have maximized $v_A(1)$ and minimized $v_A(0)$. Additionally, $v_B(1) = \frac{n}{2}+1$ and $v_B(0)=n-v_B(1)$, so if $j = \frac{n}{2}+1$ and $j' < \frac{m}{2}$, then terminate the procedure.

\textbf{Step 4.} Perform this step if and only if $j > \frac{n}{2}+1$ or $j' = \frac{m}{2}$. In this step, we only make switches with entries that are in the same column, thereby preserving the current vertex labels for all of the vertices in part $A$. Since Steps~1-3 exchange $k$ $1$s in the top half of the matrix for $0$s from row $\frac{n}{2}+1$, we have $\deg_1(u_s) = k+1 > \frac{m}{2}+2$ for $s \in \left[ c \right]$, where $c = \frac{n}{2} - k \bmod \frac{n}{2}$ is a positive integer. For these $c$ $1$-vertices in part $B$, we may replace up to $d = (k+1) - \left( \frac{m}{2}+1 \right) = k - \frac{m}{2} > 0$ of their incident $1$-edges with $0$-edges and the vertices will remain $1$-vertices. Now, $\deg_1(u_s) = k > \frac{m}{2}+1$ for $s \in \left[ \frac{n}{2} +1 - k \bmod \frac{n}{2}, \frac{n}{2}+1 \right]$, so for these $\left( 1 + k \bmod \frac{n}{2} \right)$ $1$-vertices in $B$, we may replace up to $d - 1 > 0$ of their incident $1$-edges with $0$-edges and the vertices will remain $1$-vertices. Moreover, $\deg_1(u_s) = 0$ for $s \in \left[ \frac{n}{2}+2,n \right]$.

Let $b$ be the total number of switches that we need to perform to obtain the maximal balanced index. If $j' = \frac{m}{2}$, then part $B$ should contain an unlabeled vertex and $b = \frac{mn}{2} - \left( \frac{m}{2}+1 \right) \left( \frac{n}{2}+1 \right)$; otherwise, $b = \Big[ j - \left( \frac{n}{2}+1 \right) \Big] \left( \frac{m}{2}+1 \right)$. For $a \in [cd]$, switch the 
\[ \left( 1 + (a-1)~\text{div}~d, 1 + (a-1) \bmod \left( \frac{m}{2}+1 \right) \right)\text{-entry} \]
with the 
\[ \left( \frac{n}{2}+2 + (a-1)~\text{div}~\left( \frac{m}{2}+1 \right), 1 + (a-1) \bmod \left( \frac{m}{2}+1 \right) \right)\text{-entry}. \]
This collection of switches takes the extra $d$ $1$s on row $s$, where $s \in [c]$, and exchanges them for $0$s that are in the same column but on a different row (and in the bottom half of the matrix). The row that is losing $0$s and gaining $1$s will continue to do so until the number of $1$s in that row is $\frac{m}{2}+1$, at which time the procedure simply moves down to the next row and ``starts over'' (due to the mod operator). When the number of $1$s in a row reaches $\frac{m}{2}$, the corresponding vertex changes from a $0$-vertex to an unlabeled vertex and the balanced index increases by $1$. Similarly, when the number of $1$s in a row reaches $\frac{m}{2}+1$, the unlabeled vertex becomes a $1$-vertex and the balanced index increases by $1$ again. 

Continuing, for $a \in [cd+1,b]$, switch the 
\[ \left( 1 + c + (a-1-cd)~\text{div}~(d-1), 1 + (a-1) \bmod \left( \frac{m}{2}+1 \right) \right)\text{-entry} \]
with the 
\[ \left( \frac{n}{2}+2 + (a-1)~\text{div}~\left( \frac{m}{2}+1 \right), 1 + (a-1) \bmod \left( \frac{m}{2}+1 \right) \right)\text{-entry}. \]
This collection of switches is similar to those just completed, except we are taking only $d-1$ extra $1$s on a row and exchanging them with $0$s in the same column but on a different row. The balanced index changes as before, as well. 

We started with balanced index $0$ and every switch described by the algorithm increased the balanced index by at most $1$. Although not all graphs required that every step of the algorithm completed, any early termination of the procedure was due to having already obtained $\max EBI(K_{m,n})$. At the end of Step~3, we remarked that $v_A(1) = k$ and either $v_A(0) = m-k-1$ or $m-k$, according to whether part $A$ did or did not contain an unlabeled vertex. Step~4 does not alter the labels of vertices in part $A$. Upon completion of Step~4, we find that $v_B(1) = j$ and either $v_B(0) = n-j$ (there are no unlabeled vertices in part $B$) or $v_B(0) = n-j-1$ (there is one unlabeled vertex in $B$). Thus, the construction provided by the algorithm produces the maximal balanced index for $K_{m,n}$, where $m \geq n \geq 2$ are even integers. \phantom{} \hfill $\blacksquare$
\end{proof}

The first author has written a MATLAB m-file that shows the output of each edge-label switch described in Steps~1-4. The file, called \texttt{EBI\_K.m}, can be found at \texttt{http://www.clayton.edu/faculty/craridan/code}.

\end{document}